\documentclass[10pt,a4paper]{article}

\usepackage[utf8]{inputenc}
\usepackage{amsmath}
\usepackage{amsfonts}
\usepackage{amssymb}

\usepackage{graphicx}

\usepackage{twoopt}

\usepackage{quiver}

\usepackage{amsthm}

\usepackage{comment}
\usepackage{lipsum}

\usepackage{hyperref}

\usepackage{adjustbox}

\newtheorem{theorem}{Theorem}[section]

\newtheorem{lemma}[theorem]{Lemma}

\theoremstyle{definition}
\newtheorem{example}[theorem]{Example}
\newtheorem{question}[theorem]{Question}

\theoremstyle{remark}
\newtheorem*{remark}{Remark}

\newcommand{\defaultalphabet}{X}

\newcommand{\Gp}{\operatorname{Gp}}
\newcommand{\Mon}{\operatorname{Mon}}
\newcommand{\Inv}{\operatorname{Inv}}

\newcommand{\presentation}[3]{#1 \left\langle \, #2 \mid #3 \, \right\rangle}

\newcommandtwoopt{\gpPres}[2][R][\defaultalphabet]{\presentation{\Gp}{#2}{#1}}
\newcommandtwoopt{\monPres}[2][R][\defaultalphabet]{\presentation{\Mon}{#2}{#1}}
\newcommandtwoopt{\invPres}[2][R][\defaultalphabet]{\presentation{\Inv}{#2}{#1}}

\newcommand{\subGp}[1]{\Gp \left\langle #1 \right\rangle}
\newcommand{\subMon}[1]{\Mon \left\langle #1 \right\rangle}

\newcommand{\pref}{\operatorname{pref}}

\newcommand{\freemon}[1][\defaultalphabet]{ \overline{#1}^{\ast}}

\newcommand{\dop}{ {\prime\prime} }

\newcommand{\ltr}{\lbrace t \rbrace}
\newcommand{\lttr}{\lbrace t, t^{-1} \rbrace}

\newcommand{\ouh}{\overline{U_H}}
\newcommand{\oua}{\overline{U_A}}
\newcommand{\oub}{\overline{U_B}}

\newcommand{\fmuh}{\freemon[U_H]}

\newcommand{\MRF}{\operatorname{MRF}}
\newcommand{\FG}{\operatorname{FG}}

\newcommand{\rsa}{\rightsquigarrow}
\newcommand{\lsa}{\leftsquigarrow}

\makeatletter
\providecommand{\leftsquigarrow}{%
  \mathrel{\mathpalette\reflect@squig\relax}%
}
\newcommand{\reflect@squig}[2]{%
  \reflectbox{$\m@th#1\rightsquigarrow$}%
}
\makeatother

\title{
On the submonoid membership problem for HNN extensions of free groups

\footnotetext[0]{
$2020$ \textit{Mathematics Subject Classifaction}:
20F05, 20F10, 20M05, 20M18
}

}

\author{
Jonathan Warne
\footnote{
Work undertaken while a PhD student at the University of East Anglia under the supervision of Robert Gray
}
}

\begin{document}

\maketitle

\begin{abstract}

	We study membership problems in HNN extensions of free 
groups and then apply these results to solve the word problem in certain 
families of one-relator inverse monoids.
In more detail, we consider HNN extensions where the defining isomorphism produces a bijection between subsets of a basis of the free group.
Within such HNN extensions we identify natural conditions on submonoids of this group that suffice for membership in that submonoid to be decidable.
We show that these results can then be applied to solve the prefix membership problem in certain one-relator groups which via results of Ivanov, Margolis and Meakin $(2001)$ then give solutions to the word problem for the corresponding one-relator inverse monoid.
In particular our new techniques allow us to solve the word problem in an example (Example $7.6$) from Dolinka and Gray $(2021)$ which previous methods had not been able to resolve.

\end{abstract}

\section{Introduction}

	HNN extensions of groups are well known to have wild algorithmic behaviours.
For example, Wei{\ss} proved in his PhD thesis that the conjugacy problem is undecidable in general for HNN-extensions of finitely generated free groups \cite{Wei15}, and Miller constructed examples of HNN-extensions of free groups with undecidable subgroup membership problem \cite{Mil71}.
Even when the defining isomorphism is an automorphism of a free group there are examples with undecidable submonoid membership problem (see the Burns free-by-cyclic example (\ref{eqn:undecSubmonProb}) below).
So in general, one should not expect to be able to decide membership in finitely generated submonoids of HNN extension of free groups.
On the other hand, in order to solve the word problem in various one-relator inverse monoids, by results of Ivanov, Margolis and Meakin \cite{IMM01}, it is often necessary to decide membership in certain submonoids of one-relator groups.
Approaching this problem, via the Magnus-Moldavanski{\u i} hierarchy, naturally dovetails with the question of deciding membership in certain finitely generated submonoids of HNN-extensions of free groups, which is the main topic of this paper.

	By a result of Ivanov, Margolis and Meakin \cite[Theorem 4.1]{IMM01}, to show that an E-unitary inverse monoid $\invPres[r=1]$ has decidable word problem it suffices to show that the group described by the presentation $\gpPres[r = 1]$ has decidable prefix membership problem.
While it has been shown by Dolinka and Gray \cite[Theorem 8.2]{DolGra21} that there is a group of the form $\gpPres[r=1]$ where $r$ is reduced which has undecidable prefix membership problem.
This leaves open the question of whether the prefix membership problem is decidable whenever $r$ is cyclically reduced which is particularly notable as $\invPres[r = 1]$ is E-unitary when $r$ is cyclically reduced.

	A standard approach to proving results concerning groups with a single cyclically reduced relator is to use the Magnus-Moldavanski{\u i} hierarchy.
This hierarchy can, in rough terms, be said to classify such groups by how many steps (by HNN extension) they are from being a free group.
The approach is then to show that results which hold lower in the hierarchy, in particular in free groups, extend to those higher up.

	If the group $\gpPres[r = 1]$ is free then membership in all the submonoids, including the prefix monoid, is decidable by Benois' Theorem \cite{Ben69}.
Therefore the natural next step is to check whether direct HNN extensions of free groups have decidable prefix membership problem.
However, it is already known that there are submonoids of HNN extensions of free groups which have undecidable submonoid membership problem.
In fact even free by cyclic groups can have undecidable submonoid membership problem.
Burns, Karrass and Solitar \cite{BKS87} give the following example with two possible presentations,
\begin{equation}	\label{eqn:undecSubmonProb}
	\gpPres[abba(baab)^{-1} = 1][a,b] \cong \gpPres[t^{-1}xt = xy, t^{-1}yt = y][x,y,t],
\end{equation}
which has a submonoid with undecidable membership, although in this case the prefix membership is decidable.

	A result from Dolinka and Gray \cite[Theorem 7.2]{DolGra21} of particular relevance to this line of inquiry is the following
(see the Heirarchy subsection of the Preliminaries section below for an explanation of $\Xi_w$ and $\rho_t(w)$):
\begin{theorem}
	Let $G = \gpPres[w = 1]$ be a group such that $w$ is $t$-sum zero, i.e.
the amount of occurences of $t$ and $t^{-1}$ in $w$ are equal.

Suppose that
\begin{equation*}
	H = \gpPres[\rho_t(w) = 1][\Xi_w]
\end{equation*}
is a free group.
If $w$ is $t$-prefix positive, i.e. every prefix of $w$ contains as least as many occurences of $t$ as occurences of $t^{-1}$,
then $G$ has decidable prefix membership problem.
\end{theorem}
This is, in fact, one way we can see the group in on the left of Equation \ref{eqn:undecSubmonProb} has decidable prefix membership problem (note that the presentation on the left in Equation \ref{eqn:undecSubmonProb} is $b$-sum zero and $b$-prefix positive).
The theorem leaves open the question of whether, under the same assumption that the group $H$ is free, groups defined by a relator where the word is $t$-sum zero but not $t$-prefix positive have decidable prefix membership problem.
Dolinka and Gray \cite[Example 7.6]{DolGra21} give the following simple example of a group whose prefix membership problem is left open and cannot be solved using their methods.
\begin{equation}
	\gpPres[bt^{-1}at^2bt^{-1}a = 1][a,b,t].
\end{equation}
As an application of the general results we prove in this paper, we will show that both this specific example and a family of similar groups have decidable prefix membership problem.

	We do this by developing theories of decidability of the membership problems within HNN extensions of free groups.
If we return to the example given above in Equation \ref{eqn:undecSubmonProb} we see that the presentation on the right is defined by a map from an associated subgroup generated by $x$ and $y$ to one generated by $xy$ and $y$.
Rather than this kind of HNN extension, we will instead be focusing on those where there is a bijection between two subsets of the same basis.

	In Section $2$ we describe a number of key concepts and results which will be referenced throughout the rest of the paper.
In Section $3$ we present a number of results about which words map to the same element within a particular kind of HNN extension of a free group and some algorithmic consequences of this.
In Section $4$ we show that certain submonoids of such a HNN extension of a free groups have decidable membership.
In Section $5$ we demonstrate that a certain family of groups have decidable prefix membership problem and thus that the corresponding inverse monoids have decidable word problem.

\section{Preliminaries}

\paragraph{Words and Elements}
Suppose that $X$ is an alphabet, that is a non-empty set of letters, then we use $X^\ast$ to represent the set of words written over that alphabet which includes the empty word, denoted by $1$.
If we combine this set of words with the operation of concatenation then we get the \textit{free monoid} on $X$.

	We let $\overline{X} = X \cup X^{-1}$, where $X^{-1} = \lbrace x^{-1} \mid x \in X \rbrace$.
For a word $w \in \freemon$ such that $w \equiv x_1^{\varepsilon_1} x_2^{\varepsilon_2} \ldots x_k^{\varepsilon_k}$ where $x_1 \in X$ we set
$w^{-1} \equiv x_k^{-\varepsilon_k} \ldots x_2^{-\varepsilon_2} x_1^{-\varepsilon_1}$.
This extends the obvious bijection between $X$ and $X^{-1}$ to one between $\freemon$ and itself.
If we quotient the free monoid on $\overline{X}$ by $xx^{-1} = 1$ this gives us a new object, the \textit{free group} on $X$, which we may denote by $\FG(X)$.

	We say that a set of words $U \subset \freemon$ is a \textit{free basis} for the free group $FG(X)$ if $\subGp{U} = \FG(X)$ and there is no non-trivial $U^\prime \subset U$ such that $\subGp{U^\prime} = \FG(X)$.
We call $|X|$ the \textit{rank} of the free group $\FG(X)$.
The following is a standard result (found, for instance, in \cite[Proposition I.2.7]{LynSch01})

\begin{theorem}	\label{thm:FGdefByRank}
	Let $F$ be a free group of finite rank $n$.
The free group $F$ cannot be generated by fewer than $n$ elements and if a set $U$ of $n$ elements generates $F$ then it is a free basis for $F$.
\end{theorem}

	We may define other groups by presentation, by which we mean that
\begin{equation*}
	\gpPres[u_i = v_i \, (i \in I)]
\end{equation*}
is equal to the group $\FG(X)$ quotiented by the relations $u_i = v_i$.

	We define the \textit{free inverse monoid} on $X$ as $\freemon$ quotiented by the relations $ww^{-1}w = w$ and $uu^{-1}vv^{-1} = vv^{-1}uu^{-1}$ for all $u,v,w \in \freemon$.
Similar to groups we may then define inverse monoids by presentation and say that
\begin{equation*}
	\invPres[u_i = v_i \, (i \in I)]
\end{equation*}
is the free inverse monoid on $X$ further quotiented by the given relations, for more on combinatorial inverse semigroup theory we refer the reader to \cite{Pet84} and \cite{Law98}.

\paragraph{Factorisations}

	Let $M$ be an inverse monoid.
We say that a $u \in \freemon$ is a \textit{right unit} in $M$ if there is some $v_r \in \freemon$ such that $u v_r = 1$ in $M$.
Similarly we say that it is a \textit{left unit} if there is some $v_l \in \freemon$ such that $v_l u = 1$ in $M$.
We say that $u$ is a \textit{unit} if $u$ is a left unit and a right unit.

	Let $M = \invPres[r_i = 1 \, (i \in I)]$ be a special inverse monoid.
If there are is a decomposition of the $r_i$ into component parts, i.e. there are $r_{i,j} \in \freemon$ such that $r_i \equiv r_{i,1} r_{i,2} \ldots r_{i,k_i}$ for each $i \in I$, then we call this a \textit{factorisation} of the relators and the $r_{i,j}$ the \textit{factor words}.
We say that a factorisation is \textit{unital} in $M$ if every factor word is a unit in $M$ (our language here is taken from Dolinka and Gray in \cite{DolGra21}).

	Let $w \equiv x_1 x_2 \ldots x_k$, where $x_i \in \overline{X}$ for $1 \leq i \leq k$.
We use $\pref(w)$ to denote the set of prefixes of $w$ (including $1$ and $w$ itself).
That is
\begin{equation*}
	\pref(w) = \lbrace 1, \, x_1, \, x_1 x_2, \, \ldots, \, x_1 x_2 \ldots x_k \rbrace.
\end{equation*}

	Let $G = \gpPres[r_i = 1 \, (i \in I)]$ be a group.
We call
\begin{equation*}
	P = \subGp{\bigcup_{i \in I} \pref(r_i)} \leq G
\end{equation*}
the \textit{prefix monoid} of $G$.
The prefix monoid is sensitive to the presentation of $G$ used.
Perhaps the simplest example of this is $\gpPres[aab = 1][a,b]$ and $\gpPres[aba = 1][a,b]$, both of which are isomorphic to $\mathbb{Z}$ under the mapping $a \mapsto 1$, $b \mapsto -2$ but whose prefix monoids under that same mapping are $\mathbb{Z}$ and $\mathbb{N}$ respectively.

	We say that a factorisation of the relators of $G$ is \textit{conservative} in $G$ if the prefixes of the factors generate the prefix monoid (again we are taking our language from \cite{DolGra21}).
That is 
\begin{equation*}
	\subGp{\bigcup_{i \in I} \bigcup_{1 \leq j \leq k_i} \pref(r_{i,j})}
	= \subGp{\bigcup_{i \in I} \pref(r_i)}
	\leq G
\end{equation*}
where the factorisation is $r_i \equiv r_{i,1} r_{i,2} \ldots r_{i,k_i}$ for each $i \in I$.

	These seemingly disparate notions are connected by the following result of Dolinka and Gray \cite[Theorem 3.3.(i)]{DolGra21}.
\begin{lemma}	\label{lem:unitalIsCon}
	Let $M = \invPres[r = 1]$ be a special one-relator inverse monoid and $G = \gpPres[r = 1]$ be its maximal group image.
Any factorisation of $r$ which is unital in $M$ is also conservative in $G$.
\end{lemma}
We briefly note that this can be extended to presentations with multiple relators (as the author has done in \cite{War24}).

\paragraph{HNN extensions}

	If we have a group $H = \gpPres$ with two subgroups $A$ and $B$ which are isomorphic under the map $\phi$.
Then we can construct a group
\begin{equation*}
	H \ast_{t, \phi:A, B} =
	\gpPres[R, \, t^{-1} a t = \phi(a) \quad (a \in A)][X \cup \ltr]
\end{equation*}
which we call a \textit{HNN extension} of $H$, while we call $t$ the \textit{stable letter}, $A$ and $B$ the associated subgroups and $\phi$ the associated isomorphism.

	We say that a word $w$ is in \textit{HNN reduced form} if is of the following form
\begin{equation*}
	w_0 t^{\varepsilon_1} w_1 \ldots t^{\varepsilon_k} w_k
\end{equation*}
where $\varepsilon_i \in \lbrace -1, 1 \rbrace$ for $1 \leq i \leq k$, $w_i \in \freemon$ for $0 \leq i \leq k$ and no subwords are of the form either $t^{-1} w_i t$ where $w_i \in A$ or $t w_i t^{-1} \in B$.
The following result is standard and can be proved by application of \cite[Lemma IV.2.3]{LynSch01} and Britton's Lemma.

\begin{lemma}	\label{lem:equalRedForms}

	Let $H^\ast = H \ast_{t, \phi: A \rightarrow B}$ be a HNN extension.
Then equality of between two words in HNN reduced form
\begin{equation*}
	g_0 t^{\varepsilon_1} g_1 t^{\varepsilon_2} g_2 \ldots t^{\varepsilon_n} g_n
	=
	h_0 t^{\delta_1} h_1 t^{\delta_2} h_2 \ldots t^{\delta_n} h_n
\end{equation*}
holds in the HNN extension if and only if
$n = m$,
$\varepsilon_i = \delta_i$ for all $1 \leq i \leq n$
and there exist $1 = \alpha_0, \alpha_1, \ldots, \alpha_n, \alpha_{n+1} = 1 \in A \cup B$
where $\alpha_i \in A$ if $\varepsilon_i = -1$ and $\alpha_i \in B$ if $\varepsilon_i = 1$
such that
\begin{equation*}
	h_i = \alpha_i^{-1} g_i (t^{\varepsilon_{i+1}} \alpha_{i + 1} t^{\varepsilon_{i + 1}})
\end{equation*}
for all $0 \leq i \leq n$.

\end{lemma}

	There are also known circumstances under which we can find a reduced form by algorithm, see \cite[Pages 184-185]{LynSch01} for a description of the process.

\begin{lemma}	\label{lem:redFormByAlg}

	Let $H^\ast = H \ast_{t, \phi: A \rightarrow B}$ where $H$, $A$ and $B$ are all finitely generated, $H$ has recursively enumerable word problem and membership in $A$ and $B$ within $H$ are decidable.
Then there is an algorithm which takes a word $w \in \freemon[X \cup \ltr]$ as input and returns a HNN reduced form word
\begin{equation}
	w^\prime \equiv w_0 t^{\varepsilon_1} w_1 \ldots t^{\varepsilon_n} w_n
\end{equation}
where $w_i \in \freemon$ for $0 \leq i \leq n$, $\varepsilon_i \in \lbrace -1, 1 \rbrace$ for $1 \leq i \leq n$ and $w = w^\prime$ in $H^\ast$.

\end{lemma}

\paragraph{Hierarchy}

	For a given letter $y \in Y$ we can define a function $\sigma_y(w): \freemon \rightarrow \mathbb{Z}$, which returns the number of appearances of $y$ in $w \in \freemon$ less the number of appearances of $y^{-1}$ in $w$.
We call this total the \textit{exponent sum} of $y$ in $w$ and say that $w$ has $y$ \textit{exponent sum zero} if $\sigma_x(w) = 0$.
For instance $\sigma_y(xy^2x^{-1}y^{-1}zy^{-1}) = 2 - 1 - 1 = 0$.

	Let $X$ be an alphabet, then we can define a second alphabet,
\begin{equation*}
	\Xi_X = \lbrace x_i \, \mid x \, \in \overline{X}, \, i \in \mathbb{Z} \rbrace
\end{equation*}
and a function
\begin{equation*}
	\rho_t : \lbrace w \in (\overline{X} \cup \lttr)^\ast \mid 
	w \text{ is cyclically reduced} \text{ and } \sigma_t(w) = 0 \rbrace
	\rightarrow \Xi_X
\end{equation*}
for a given $t$.
This defined as follows;
let $Y = X \cup \ltr$,
$w \equiv y_1 y_2 \ldots y_n$ where $y_i \in Y$, for $1 \leq i \leq n$,
and finally let $p_i \equiv y_1 y_2 \ldots y_i$ then if $y_i \in \lttr$ it is mapped to $1$ and if $y_i \equiv x \in X$ it is mapped to $x_{-\sigma_t(p_i)} \in \Xi_X$.
For example
\begin{equation*}
	\rho_t(yt^{-1}xt^2yt^{-1}x) \equiv y_0 x_1 y_{-1} x_0.
\end{equation*}
The inverse of this function simply sends each $x_i$ to $t^{-i} x t^i$, for instance
\begin{equation*}
	\rho_t(y_0 x_1 y_{-1} x_0)
	\equiv (t^{-0} y t^0) (t^{-1} x t^1) (t^{-(-1)} y t^{-1}) (t^{-0} x t^0).
	= yt^{-1}xt^2yt^{-1}x.
\end{equation*}

	For any particular word $w$ we can define $\mu_x$ and $m_x$ to be, respectively, the minimum and maximum values of $i$ such that $x_i$ appears in $\rho_t(w)$.
If $x$ does not appear at all in $w$ then we say that $\mu_x = 0 = m_x$.
For instance in our above example $\mu_y = -1$ and $m_y = 0$.

	Using these we can define a subset of $\Xi_X$,
\begin{equation*}
	\Xi_w = \lbrace x_i \in \Xi_X \mid \mu_x \leq i \leq m_x \rbrace
\end{equation*}
so in our running example $\Xi_w = \lbrace x_0, x_1, y_{-1}, y_0 \rbrace$.
However in general it need not be the case that all the letters of $\Xi_X$ appear in $\rho_t(w)$, for instance if $X = \lbrace x, y, z \rbrace$ and $w = yxt^{-2}xt^2$ we have $\Xi_w = \lbrace x_0, x_1, x_2 \rbrace$ and $\rho_t(w) = \lbrace x_0, x_1, x_2, y_0, z_0 \rbrace$.

	With these ideas in place we introduce the following result, originally due to Moldavanski\u i, a proof can be found within \cite[Theorem IV.5.1]{LynSch01} which uses the methods of McCool and Schupp.
\begin{theorem}	\label{thm:HNNextByRho}
	Let $w \in \freemon$ be a word with $t$ exponent sum zero such that $\rho_t(w)$ is cyclically reduced and let $G = \gpPres[w = 1]$ be a group.
Then $G$ is an HNN extension of the group
\begin{equation*}
	H = \gpPres[\rho_t(w) = 1][\Xi_w]
\end{equation*}
where the associated subgroups $A$ and $B$ are free and generated by $\Xi_w \setminus \lbrace x_{m_x} \mid x \in X \rbrace$ and $\Xi_w \setminus \lbrace x_{\mu_x} \mid x \in X \rbrace$ respectively and the associated isomorphism is defined by $\phi(x_i) = x_{i+1}$ for $x \in X$ and $\mu_x \leq i < m_x$.
\end{theorem}

	This is easiest to understand through examples

\begin{example}
	If $X = \lbrace x, y \rbrace$ and $w \equiv yt^{-1}xt^2yt^{-1}x$ then
\begin{equation*}
	H = \gpPres[y_0 x_1 y_{-1} x_0 = 1][x_0, x_1, y_{-1}, y_0],
\end{equation*}
$A = \subGp{x_0, y_{-1}}$, $B = \subGp{x_1, y_0}$ and $\phi$ is the extension of the function sending $x_0$ to $x_1$ and $y_{-1}$ to $y_0$.
\end{example}

\begin{example}
	If $X = \lbrace x, y, z \rbrace$ and $w \equiv yxt^{-2}xt^2x$ then 
\begin{equation*}
	H = \gpPres[x_0 x_2 = 1][ x_0, x_1, x_2, y_0, z_0],
\end{equation*}
$A = \subGp{x_0, x_1}$, $B = \subGp{x_1, x_2}$ and $\phi$ is the extension of the function sending $x_0$ to $x_1$ and $x_1$ to $x_2$.
\end{example}

\section{Sequences of Elementary Operations}

	In this section we will examine a particular class of HNN extensions of free groups.
We will show that these preserve a degree of freeness in the sense that they have a consistent notion of cancellation within the reduced forms.

	Throughout this section we will assume that $H$ is a free group over an alphabet $X$ which has a free basis $U_H \subset \freemon$.
We will also assume that $A$ and $B$ are isomorphic subgroups of $H$ which are generated by $U_A, U_B \subseteq U_H$ respectively and that the isomorphism between them, which we call $\phi$, produces a bijection between $\oua$ and $\oub$.
Finally we will also be assuming that $H^\ast = H \ast_{t, \phi:A \rightarrow B}$.

	We begin by noting that such groups always have decidable word problem by \cite[Corollary IV.2.2]{LynSch01}.
Next make the following observations
\begin{lemma}	\label{lem:monVMgpVG}
	Let $V_G \subseteq U_H$ and $V_M \subseteq \ouh$ be non-empty sets.
Then we have the following:
\begin{enumerate}
	\item The sets $V_M \cap \overline{V_G}$ and $V_M \cap \subGp{V_G}$ respresent the same elements in $H$
	\item The sets $(V_M \cap \overline{V_G})^\ast$ and $\subMon{V_M} \cap \subGp{V_G}$ represent the same elements in $H$
\end{enumerate}
\end{lemma}

\begin{proof}
	$(1)$:
Suppose there exists some $v \in (V_M \cap \subGp{V_G}) \setminus (V_M \cap \overline{V_G})$.
This may easily be found to be equivalent to $v \in (V_M \cap \subGp{V_G}) \setminus \overline{V_G}$.
That $v \in \subGp{V_G}$ implies that there is some word $w \in \freemon[V_G]$ such that $w = v$ in $H$.
However $v \in V_M \setminus \overline{V_G}$, so this implies that an element of the free basis $U_H$ can be written in terms of different members of the free basis in $H$, which is a contradiction.
Therefore $V_M \cap \subGp{V_G}$ represent a subset of the elements that $V_M \cap \overline{V_G}$ do in $H$.
The converse is obvious from the definitions.

	$(2)$:
Suppose that $x \in \subMon{V_M} \cap \subGp{V_G}$.
That $x \in \subMon{V_M}$ implies there is a word $w_M \in V_M^\ast$ such that $w_M = x$ in $H$.
Similarly, that $x \in \subGp{V_G}$ implies there is a word $w_G \in \freemon[V_G]$ such that $w_G = x$ in $H$.
As $w_H$ and $w_G$ are words written over $\ouh$, $U_H$ is a free basis for $H$ and $w_M = x = w_G$ in $H$ we may conclude that both $w_M$ and $w_G$ reduce over $U_H$ to the same word written over $\ouh$.
The resulting word, $w$, will satisfy $w \in (V_M \cap \overline{V_G})^\ast$.
Thus we may conclude that $\subMon{V_M} \cap \subGp{V_G} \subseteq (V_M \cap \overline{V_G})^\ast$.
The converse is obvious from the defintions.

\end{proof}

\begin{remark}
	In the notation used here $\emptyset^\ast$ consists solely of the empty word.
Hence for part $(2)$ in the case where the intersection of $V_M$ and $\overline{V_G}$ is empty the set $(V_M \cap \overline{V_G})^\ast$ represents only the identity element in $H$.
\end{remark}

	If we let $V_G = U_A$ and $V_M = U_H$ then part $(1)$ of Lemma \ref{lem:monVMgpVG} tells us that $U_H \cap A = U_H \cap \oua = U_A$.
Similarly we may conlude that $U_H \cap B = U_B$.

	Suppose we have a word $w$ written over $\overline{X} \cup \lttr$ in the following form
\begin{equation*}
	w \equiv w_0 t^{n_1} w_1 \ldots t^{n_k} w_k,
\end{equation*}
where $w_i \in \freemon$ and $n_i \in \mathbb{Z}$.
As $U_H$ is a free basis for $\FG(X)$ every $w_i$ may be uniquely rewritten as a word $w_i^\prime \in \ouh^\ast$ such that $w_i^\prime = w_i$ in $H$ and that $w_i^\prime$ is reduced in terms of $U_H$.
Applying this to every $w_i$ of $w$ yields a unique rewriting
\begin{equation*}
	w^\prime \equiv w_0^\prime t^{n_1} w_1^\prime \ldots t^{n_k} w_k^\prime.
\end{equation*}
Further, this is such that if our initial $w$ is in HNN reduced form in $H^\ast$ then so is $w^\prime$.
Thus together with Lemma \ref{lem:redFormByAlg} we can algorithmically rewrite any word over $X \cup \lttr$ as a word over $\ouh \cup \lttr$ which is in reduced form in $H^\ast$.

	 We list three \textit{elementary operations}, which can be applied to a given word written over $\overline{U_H} \cup \lttr$ to produce another word written over $\overline{U_H} \cup \lttr$.

\begin{itemize}
	\item \textit{Elementary Addition}, the insertion of some pair $u u^{-1}$ into the word, where $u \in \ouh$, in such a way that the result is still written over $\ouh \cup \lttr$;
	\item \textit{Elementary Cancellation}, the removal of some pair $u u^{-1}$ from the word, where $u \in \ouh$;
	\item \textit{Elementary Semicommutation}, either substitution of $u_a t$ for $t u_b$ or vice versa or the substitution of $t^{-1} u_a$ for $u_b t^{-1}$ or vice versa, where $u_a \in U_A$, $u_b \in U_B$ and $\phi(u_a) = u_b$.
\end{itemize}

	We will now demonstrate that these are sufficient under our assumptions to serve as the discrete units to describe a transformation between any two words both equal and in HNN reduced form in $H^\ast$.

\begin{lemma}	\label{lem:equalMeansSeqOfElOps}

	Let $w, w^\prime \in \freemon[U_H \cup \ltr]$ be two words which are in HNN reduced form in $H^\ast$.
The words $w$ and $w^\prime$ are equal in $H^\ast$ if and only if there is a sequence of elementary operations which turns $w$ into $w^\prime$.

\end{lemma}

\begin{proof}

	Let $Y$ be a set of letters of size $|U_H|$.
As $X$ and $U_H$ are both bases of $H$, we know that $|X| = |U_H|$ so $|Y| = |X|$ also.
Further let $\theta: U_H \rightarrow  Y$ be a bijection.
As $U_H$ is a basis for $H = \FG(X)$ this extends to a bijection $\theta: \FG(X) \rightarrow \FG(Y)$.

	We may write that
\begin{equation*}
	H^\ast = \gpPres[t^{-1}ut = \phi(u) \quad (u \in U_A)][X, t].
\end{equation*}
If we let $Y_A = \theta(U_A)$ and $Y_B = \theta(U_B)$ then we may likewise define the following group
\begin{equation*}
	K = \gpPres[s^{-1}\theta(u)s = \theta(\phi(u)) \quad (u \in U_A)][Y, s]
\end{equation*}
This $K$ is isomorphic to $H^\ast$ under an extension of $\theta$ which also sends $t$ to $s$.

	That $\theta$ is an isomorphism implies $\theta(w) = \theta(w^\prime)$ in $K$ if and only if $w = w^\prime$ in $H^\ast$.
In $K$ it is obvious that two words in HNN reduced form are equal if and only if some sequence of composed of: insertions of pairs of the form $yy^{-1}$ where $y \in \overline{Y}$; cancellations of pairs of the form $yy^{-1}$ where $y \in \overline{Y}$ and substitutions of the form $y_a t$ for $t y_b$ or $t^{-1} y_a$ for $y_b t^{-1}$ or vice versa where $\phi( \theta^{-1}(y_a)) = \theta^{-1}(y_b)$.
Such a sequence easily translates to a sequence of elementary operations in $H^\ast$.
Therefore $w = w^\prime$ in $H^\ast$ if and only if there is a sequence of elementary operations between $w$ and $w^\prime$.
\end{proof}

\begin{remark}
	At this stage it is worth observing that all the remaining results in this section and Theorem \ref{thm:submonMemInHast} could be proved in $K$ and then transferred by the isomorphism $\theta^{-1}$ to $H^\ast$.
However, Theorem \ref{thm:wtwWords}, which is one of the main results of the paper, needs to work even when the factorisation is into subwords which are not letters.
With this in mind, we feel it is clearer to assume throughout that our elementary operations are acting on the words of a free basis rather than letters.
\end{remark}

It is usual practice, as we have done so far, to represent a word $w$ in reduced HNN form in the following form $w_0 t^{\varepsilon_1} w_1 \ldots t^{\varepsilon_n} w_n$, where each $w_i \in \freemon$ may be empty but the $\varepsilon_i$ are either $1$ or $-1$.
However for the remainder of this section we will instead use
\begin{equation*}
	w \equiv t^{n_0} u_1 t^{n_1} \ldots u_j t^{n_j} \ldots u_k t^{n_k},
\end{equation*}
where $u_i \in U_H$ for $1 \leq k$ and $n_j \in \mathbb{Z}$ for $0 \leq j \leq k$, thus prioritising consistent labelling of the individual elements of the free basis over consistent labelling of the individual stable letters.

\begin{remark}
	Though all semicommutations can be described as a replacement of some $t^{n_{j-1}} u_j t^{n_j}$ with either $t^{n_{j-1} + 1} \phi(u_j) t^{n_j - 1}$ or $t^{n_{j-1} - 1} \phi^{-1}(u_j) t^{n_j + 1}$ this \textbf{does not} mean that all such replacements represent semicommutations.
For instance, the replacement of $t^0 u_a t^0$ with $t \phi(u_a) t^{-1}$, for some $u_a \in \oua$ does not represent an elementary semicommutation.
Note that the former is in HNN reduced form while the latter is not.
\end{remark}

	We will use $v \sim v^\prime$ to indicate that there is a finite sequence of elementary semicommutations between $v$ and $v^\prime$.

\begin{lemma}	\label{lem:basisWordShift}

	Suppose we have two words, $w, w^\prime \in \freemon[U_H \cup \ltr]$, in HNN reduced form that are written
\begin{equation*}
	w \equiv t^{n_0} u_1 t^{n_1} \ldots u_j t^{n_j} \ldots u_k t^{n_k}
\end{equation*}
and
\begin{equation*}
	w^\prime \equiv t^{n_0^\prime} u_1^\prime t^{n_1^\prime} \ldots u_j^\prime t^{n_j^\prime} \ldots u_{k^\prime}^\prime t^{n_{k^\prime}^\prime}.
\end{equation*}
where $u_j, u_j^\prime \in \ouh$ and $n_j, n_j^\prime \in \mathbb{Z}$,

	If $w \sim w^\prime$ then $k = k^\prime$ and $u_j^\prime \equiv \phi^{z_j} ( u_j )$ where $z_j = \Sigma_{i=0}^{j-1} (n_i^\prime - n_i)$, for $1 \leq j \leq k$.
Furthermore this means that the set $W = \lbrace w^\prime \mid w^\prime \sim w \rbrace$ is finite and can be found by algorithm.

\end{lemma}

\begin{proof}

	As semicommutation does not change the number of basis words present we immediately have that $k = k^\prime$.

	Suppose that difference between $w$ and $w^\prime$ is a single elementary semicommutation centred around the $j$th position of $w$.
Then looking at the definition we can see that it must have one of two effects,either $t^{n_{j-1}} u_j t^{n_j}$ is replaced by $t^{n_{j-1} + 1} \phi(u_j) t^{n_j - 1}$ or by $t^{n_{j-1} - 1} \phi^{-1}(u_j) t^{n_j + 1}$.

	We can see that $z_i = z_{i-1} + n_i^\prime - n_i$ for $1 < i \leq k$.
Therefore $z_i = z_{i-1} = 0$ for all $i < j-1$ as we have not changed anything left of $t^{n_{j-1}}$.
Looking at our two cases we see that $n_{j-1}^\prime - n_{j-1}$ will be equal to $-1$ and $1$ respectively as required.
Further as $n_{j-1} + n_j = n_{j-1}^\prime + n_j^\prime$ in both cases we can see that
\begin{equation*}
	z_j = z_{j-2} + (n_{j-1}^\prime - n_{j-1}) + (n_j^\prime - n_j) = z_{j-2} = 0.
\end{equation*} 
As we have also not changed anything to the right of $t^{n_j}$ either we can also conclude that $z_i = 0$ for $j < i$ as well.
Thus we see that the formula correctly finds that $u_i^\prime = \phi^0(u_i) \equiv u_i$ when $i \neq j-1$.

	Suppose that we have a sequence of $K$ elementary semicommutations between $w$ and $w^\prime$ for which the lemma's statement holds.
Further suppose we have $w^\dop$ which is a single elementary semicommutation away from $w^\prime$.
We say that
\begin{equation*}
	w^\dop \equiv t^{n_0^\dop} u_1^\dop t^{n_1^\dop} \ldots u_j^\dop t^{n_j^\dop} \ldots u_k^\dop t^{n_k^\dop}.
\end{equation*}
We know by the previous section of the proof that $u_j^\dop \equiv \phi^{z_j^\prime}(u_j^\prime)$ where $z_j^\prime = \Sigma_{i=0}^{j-1} (n_i^\dop - n_i^\prime)$.
Further we know by inductive assumption that $u_j^\prime \equiv \phi^{z_j}(u_j)$ where $z_j = \Sigma_{i=0}^{j-1} (n_i^\prime - n_i)$.
Combining this we can see that $u_j^\dop \equiv \phi^{z_j + z_j^\prime}(u_j)$ and that
\begin{equation*}
	z_j + z_j^\prime = 
	\Sigma_{i=0}^{j-1} (n_i^\prime - n_i)
	+
	\Sigma_{i=0}^{j-1} (n_i^\dop - n_i^\prime)
	=
	\Sigma_{i=0}^{j-1} (n_i^\dop - n_i),
\end{equation*}
which means the sequence between $w$ and $w^\dop$, and therefore all sequences of length $K+1$, satisfy the lemma's statement.

	Thus we can conclude by induction that the first part of the Lemma's statement holds for sequences of any finite length.

	Let $N$ be the total number of instances of $t$ and $t^{-1}$ in $w$.
This means that $N = \Sigma_{i=0}^k |n_i|$.
We have established above that the entirety of any word $w^\prime \in W$ can be determined by knowing the $n_i^\prime$ values.
As by assumption $w$ and $w^\prime$ are both in HNN reduced form the total number of instances of $t^{-1}$ and $t$ must be the same in both.
Therefore $N = \Sigma_{i=0}^k |n_i^\prime|$ also, for every $w^\prime \in W$.
There are only a finite number of ways to partition $N$ into $n_0^\prime, n_1^\prime, \ldots, n_k^\prime$ and then by multiplying this figure by a finite amount, to account for choice of sign, it is possible to produce a finite limit on the number of possible sets of choices for $n_i^\prime$.
Thus there a finite number of $w^\prime$ such that $w^\prime \sim w$.
As only a finite number of elementary semicommutations are possible from any particular $w^\prime$ it is possible for an algorithm to find every member of $W$ by an exhaustive search.

\end{proof}

\begin{example}
	Let $H = \gpPres[abc = 1][a,b,c,t]$, $U_H = \lbrace a, b \rbrace$, $U_A = \lbrace a \rbrace$, $U_B = \lbrace b \rbrace$ and $\phi$ which extends $a \mapsto b$.
This gives us the HNN extension $H^\ast = \gpPres[abc = 1, t^{-1}at = b][a,b,c,t]$.
Consider the word $w \equiv bt^{-2}at^{-1}a$.
We may write this as $w \equiv t^{0} b t^{-2} a t^{-1} a t^{0}$.
One semicommutation transforms this to $t^{0} b t^{-2 + 1} \phi(a) t^{-1 - 1} a t^{0} = t^{0} b t^{-1} b t^{-2} a t^{0}$.
While a second consectutive semicommutation gives $t^{0 - 1} \phi^{-1}(b) t^{-1 + 1} b t^{-2} a t^{0} = t^{-1} a t^{0} b t^{-2} a t^{0}$.
\end{example}

	Now that we have seen there is a limit to the extent a word, and in particular its basis words, can change under a sequence of elementary semicommutations we will now examine to what extent elementary cancellations are invariant under sequences of elementary semicommutations.

	Before this we define the following pieces of notation:

	Let $w, v \in \freemon[U_H \cup \ltr]$ be two words in HNN reduced form.
We will use $w \rightarrow_m v$ to specify that there is a single elementary cancellation between the $m$th and $m+1$th basis words of $w$ which transforms it into $v$.
That is if
\begin{equation*}
	w \equiv t^{n_0} u_1 t^{n_1}
	\ldots
	t^{n_{m-1}} u_m t^{n_m} u_{m+1} t^{n_{m+1}}
	\ldots
	u_k t^{n_k}
\end{equation*}
then $u_{m+1} \equiv u_m^{-1}$, $n_m = 0$ and
\begin{equation*}
	v \equiv t^{n_0} u_1 t^{n_1}
	\ldots
	t^{n_{m-1} + n_{m+1}}
	\ldots
	u_k t^{n_k}.
\end{equation*}

\begin{lemma}	\label{lem:mCancelSemicommutingSquare}

	Let $w, w^\prime, v, v^\prime \in (U_H \cup \lbrace t, t^{-1} \rbrace )^\ast$ be words in HNN reduced form.
If $v \leftarrow_m w \sim w^\prime \rightarrow_m v^\prime$ then $v \sim v^\prime$.

\end{lemma}

\begin{proof}

	As $w \sim w^\prime$ there must be a finite sequence $w_0, w_1, \ldots, w_k$ such that $w \equiv w_0$, $w^\prime \equiv w_k$ and the difference between $w_i$ and $w_{i+1}$ is a single elementary semicommutation, for $0 \leq i < k$.
Suppose that
\begin{equation*}
	w_i \equiv t^{n_{i,0}} u_{i,1} t^{n_{i,1}}
	\ldots
	t^{n_{i,m-1}} u_{i,m} t^{n_{i,m}} u_{i,m+1} t^{n_{i,m+1}}
	\ldots
	u_{i,k} t^{n_{i,k}}
\end{equation*}
and
\begin{equation*}
	v_i \equiv t^{n_{i,0}} u_{i,1} t^{n_{i,1}}
	\ldots
	t^{n_{i,m-1} + n_{i,m} + n_{i,m+1}}
	\ldots
	u_{i,k} t^{n_{i,k}}.
\end{equation*}
Under this definition $v \equiv v_0$ and $v^\prime \equiv v_k$.

	Further we claim that, for $0 \leq i < k$, either $v_i$ and $v_{i+1}$ are equal or there is a single elementary semicommutation between them.
The semicommutation between $w_i$ and $w_{i+1}$ can be represemted by the replacement of $t^{n_{i,j-1}} u_{i,j} t^{n_{i,j}}$ with $t^{n_{i,j-1} + \varepsilon} \phi^\varepsilon (u_{i,j}) t^{n_{i,j} - \varepsilon}$ for some $1 \leq j \leq k$ and $\varepsilon \in \lbrace -1, 1 \rbrace$.
If $0 \leq j < m$ or $m + 1 < j \leq k$ then the same replacement can be performed to get from $v_i$ to $v_{i+1}$ by a single elementary semicommutation.
If $j = m$ then $v_i \equiv v_{i+1}$ as $(n_{i,m-1} + \varepsilon) + (n_{i, m} - \varepsilon) + n_{i,m+1} = n_{i,m-1} + n_{i,m} + n_{i,m+1}$.
Similarly $v_i \equiv v_{i+1}$ if $j = m+1$.

	As each entry in the sequence $v_0, v_1, \ldots, v_k$ is either equal to or a single semicommutation away from the next one it follows that $v \equiv v_0 \sim v_k \equiv v^\prime$.

	We include a representation of this proof in Figure \ref{fig:mEqual}.

\end{proof}

\begin{figure}
\[\begin{tikzcd}
	{w_0} & {w_1} & {w_2} & {w_3} & \ldots & {w_{\ell}} \\
	\\
	v & {v_1} & {v_2} & {v_3} & \ldots & {v_{\ell}}
	\arrow["\sim"{description}, draw=none, from=1-1, to=1-2]
	\arrow["m"', from=1-1, to=3-1]
	\arrow["\sim"{description}, draw=none, from=1-2, to=1-3]
	\arrow[dashed, from=1-2, to=3-2]
	\arrow["\sim"{description}, draw=none, from=1-3, to=1-4]
	\arrow[dashed, from=1-3, to=3-3]
	\arrow["\sim"{description}, draw=none, from=1-4, to=1-5]
	\arrow[dashed, from=1-4, to=3-4]
	\arrow["\sim"{description}, draw=none, from=1-5, to=1-6]
	\arrow["m", from=1-6, to=3-6]
	\arrow["\equiv"{description}, draw=none, from=3-1, to=3-2]
	\arrow["\sim"{description}, draw=none, from=3-2, to=3-3]
	\arrow["\sim"{description}, draw=none, from=3-3, to=3-4]
	\arrow["\equiv"{description}, draw=none, from=3-4, to=3-5]
	\arrow["\sim"{description}, draw=none, from=3-5, to=3-6]
\end{tikzcd}\]

\caption{	\label{fig:mEqual}
The solid arrows represent cancellation at $m$, same as in the text.
The dashed arrows denote the ``cancellation'' between the $w_i$ and their respective $v_i$.
The particular sequence of $\equiv$ and $\sim$ between the $v_i$ is purely illustrative.
}
\end{figure}
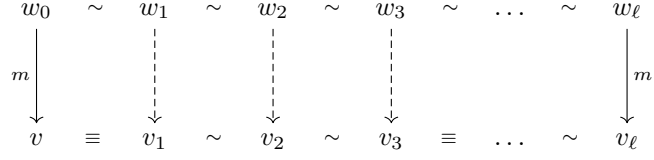

	Let $w, v \in \freemon[U_H \cup \ltr]$ be words in HNN reduced form.
We will use $w \rightsquigarrow_m v$ to indicate that there are HNN reduced form words $w^\prime, v^\prime \in \freemon[U_H \cup \ltr]$ such that $w \sim w^\prime \rightarrow_m v^\prime \sim v$.

\begin{lemma}	\label{lem:rsaMeansNiceIndices}

	Let $v, w \in \freemon[( U_H \cup \ltr )]$ be words in HNN reduced form such that $w \rightsquigarrow_m v$.
If
\begin{equation*}
	w \equiv t^{n_0} u_1 t^{n_1}
	\ldots
	t^{n_{m-1}} u_m t^{n_m} u_{m+1} t^{n_{m+1}}
	\ldots
	u_k t^{n_k}
\end{equation*}
where $u_j \in \ouh$ and $n_j \in \mathbb{Z}$, then $u_{m+1} \equiv \phi^{n_m} (u_m^{-1})$.

\end{lemma}

\begin{proof}

	That $w \rightsquigarrow_m v$ implies that there are $w^\prime, v^\prime \in \freemon[U_H \cup \ltr]$ in HNN reduced form such that $w \sim w^\prime \rightarrow_m v^\prime \sim v$.
We may write
\begin{equation*}
	w^\prime \equiv t^{n_0^\prime} u_1^\prime t^{n_1^\prime} \ldots u_j^\prime t^{n_j^\prime} \ldots u_{k}^\prime t^{n_{k}^\prime}.
\end{equation*}
and as cancellation is possible at $m$ we know that $(u_m^\prime)^{-1} \equiv u_{m+1}^\prime$ and $n_{m}^\prime = 0$.
Further we know by Lemma \ref{lem:basisWordShift} that $u_j^\prime \equiv \phi^{z_j} (u_j)$ where $z_j = \Sigma_{i=0}^{j-1} (n_i^\prime - n_i)$.

	Combining these facts, we have that
\begin{equation*}
	\phi^{z_m} (u_m^{-1})
	\equiv (u_m^\prime)^{-1}
	\equiv u_{m+1}^\prime
	\equiv \phi^{z_{m+1}} (u_{m+1})
\end{equation*}
and thus $u_{m+1} \equiv \phi^{z_m - z_{m+1}} (u_m^{-1})$.
Moreover, we have that
\begin{equation*}
	z_m - z_{m+1} = n_m - n_m^\prime = n_m - 0 = n_m
\end{equation*}
and therefore that $u_{m+1} \equiv \phi^{n_m} (u_m^{-1})$.

\end{proof}

	Two simple consequences of this are:

\begin{lemma}	\label{lem:WrsaVmeansWrsaNiceV}

	Let $w \in ( U_Q \cup \lttr)^\ast$, where $U_Q \subseteq \ouh$, and $v \in \freemon[( U_H \cup \ltr )]$ be words in HNN reduced form such that $w \rightsquigarrow_m v$.
Then there is a word $v^\prime \in ( U_Q \cup \lttr )^\ast$ in HNN reduced form such that $w \rightsquigarrow_m v^\prime$.

\end{lemma}

\begin{proof}

	By Lemma \ref{lem:rsaMeansNiceIndices} we know that we may write $w$ in the form
\begin{equation*}
	t^{n_0} u_1 t^{n_1}
	\ldots
	t^{n_{m-1}} u_m t^{n_m} \phi^{n_m}(u_m^{-1}) t^{n_{m+1}}
	\ldots
	u_k t^{n_k}.
\end{equation*}
By a finite sequence of elementary semicommutations we can transform this into
\begin{equation*}
	t^{n_0} u_1 t^{n_1}
	\ldots
	t^{n_{m-1}} u_m u_m^{-1} t^{n_m + n_{m+1}}
	\ldots
	u_k t^{n_k}
\end{equation*}
which after an elementary cancellation gives
\begin{equation*}
	t^{n_0} u_1 t^{n_1}
	\ldots
	t^{n_{m-1} + n_m + n_{m+1}}
	\ldots
	u_k t^{n_k}.
\end{equation*}
This word is written over a subset of the basis words used in $w$ as the $u_j$ where $j \neq m, m+1$ are untouched.
Thus if we designate it $v^\prime$ it satisfies our conditions.

\end{proof}

\begin{lemma}	\label{lem:wDopExists}

	Let $v, v^\prime, w, w^\prime \in \freemon[U_H \cup \ltr]$ be words in HNN reduced form such that $v \leftarrow_m w \sim w^\prime \rightarrow_{m^\prime} v^\prime$.
There is some $w^\dop \in \freemon[U_H \cup \ltr]$ such that $w \sim w^\dop$ and elementary cancellations are possible at both $m$ and $m^\prime$ in $w^\dop$.

\end{lemma}

\begin{proof}

	We may assume without loss of generality that $m \leq m^\prime$. 
If we write
\begin{equation*}
	w \equiv 
	t^{n_0} u_0 t^{n_1}
	\ldots 
	u_m t^{n_m} u_{m+1}
	\ldots 
	u_{m^\prime} t^{n_{m^\prime}} u_{m^\prime + 1}
	\ldots
	u_k t^{n_k}
\end{equation*}
then we know by Lemma \ref{lem:rsaMeansNiceIndices} that $w \rightsquigarrow_{m^\prime} v^\prime$ implies that $u_{m^\prime+1} \equiv \phi^{n_{m^\prime}}(u_{m^\prime}^{-1})$.
Therefore as we already know that $n_m = 0$ and $u_{m+1} \equiv u_m^{-1}$ by $v \leftarrow_m w$ we can write
\begin{equation*}
	w \equiv 
	t^{n_0} u_0 t^{n_1}
	\ldots 
	u_m  u_m^{-1}
	\ldots
	u_{m^\prime} t^{n_{m^\prime}} \phi^{n_{m^\prime}}(u_{m^\prime}^{-1}) t^{n_{m^\prime + 1}}
	\ldots
	u_k t^{n_k}
\end{equation*}
which is a finite sequence of elementary semicommutations away from
\begin{equation*}
	w^\dop \equiv 
	t^{n_0} u_0 t^{n_1}
	\ldots 
	u_m  u_m^{-1}
	\ldots
	u_{m^\prime} u_{m^\prime}^{-1}  t^{n_{m^\prime} + n_{m^\prime + 1}}
	\ldots
	u_k t^{n_k}
\end{equation*}
which has possible cancellations at both $m$ and $m^\prime$.

\end{proof}

	Using this we can strengthen Lemma \ref{lem:mCancelSemicommutingSquare}.

\begin{lemma}	\label{lem:smallSinkTechnical}

	Suppose we have words $v, v^\prime, w, w^\prime \in \freemon[U_H \cup \ltr]$ in reduced HNN form such that $v \leftarrow_m w \sim w^\prime \rightarrow_{m^\prime} v^\prime$.
Either $\mid m - m^\prime \mid \leq 1$ in which case $v \sim v^\prime$ or $\mid m - m^\prime \mid > 1$ in which case there are some $v^\flat, v^\dop, v^\sharp \in \freemon[U_H \cup \ltr]$ in HNN reduced form such that $v \sim v^\flat \rightarrow_m v^\dop \leftarrow_{m^\prime} v^\sharp \sim v^\prime$.

\end{lemma}

\begin{proof}

	Suppose that $\mid m - m^\prime \mid = 0$.
Then $m = m^\prime$ and this case has already been taken care of by Lemma \ref{lem:mCancelSemicommutingSquare}.

\begin{figure}
\[\begin{tikzcd}
	{w_0} && {w_{-z}} && {w_\ell} \\
	\\
	v && {v^{\prime \prime}} && {v^\prime}
	\arrow[squiggly, no head, from=1-1, to=1-3]
	\arrow[curve={height=-12pt}, squiggly, no head, from=1-1, to=1-5]
	\arrow["m"', from=1-1, to=3-1]
	\arrow["{m^\prime}", curve={height=-12pt}, from=1-3, to=3-3]
	\arrow["m"', curve={height=12pt}, from=1-3, to=3-3]
	\arrow["{m^\prime}", from=1-5, to=3-5]
	\arrow[squiggly, no head, from=3-1, to=3-3]
	\arrow[squiggly, no head, from=3-3, to=3-5]
\end{tikzcd}\]
\caption{	\label{fig:mDifOne}
The squiggly lines represent a sequence of elementary semicommutations in the same way $\sim$ does in the body of the text.
}
\end{figure}

	Suppose that $\mid m - m^\prime \mid = 1$.
By Lemma \ref{lem:wDopExists} we know that there exists some $w^\dop \in \freemon[U_H \cup \ltr]$ such that cancellations at both $m$ and $m^\prime$ are possible and $w \sim w^\dop$.
We assume without loss of generality that $m^\prime = m + 1$.
Thus we know that if we write $w^\dop \equiv t^{n_0} u_1 t^{n_1} \ldots u_m t^{n_m} u_{m+1} t^{n_{m+1}} u_{m+2} \ldots u_k t^{n_k}$ then $u_m \equiv u_{m+1}^{-1} \equiv u_{m+2}$ and $n_{m} = 0 = n_{m+1}$.
Therefore the cancellations at $m$ and $m^\prime$ both produce the same result which we call $v^\dop$.
This means that $v \leftarrow_m w \sim w^\dop \rightarrow_m v^\dop$ and so by Lemma \ref{lem:mCancelSemicommutingSquare} $v \sim v^\dop$.
Similarly we can argue that $v^\dop \sim v^\prime$.
As $\sim$ is transitive we have that $v \sim v^\prime$.
See Figure \ref{fig:mDifOne} for an illustration of this.

\begin{figure}
\[\begin{tikzcd}
	{w_0} &&& {w_{-z}} &&& {w_\ell} \\
	\\
	v && {v^\flat} && {v^\sharp} && {v^\prime} \\
	\\
	&&& {v^{\prime \prime}}
	\arrow[squiggly, no head, from=1-1, to=1-4]
	\arrow["m"', from=1-1, to=3-1]
	\arrow["m"', from=1-4, to=3-3]
	\arrow["{m^\prime}", from=1-4, to=3-5]
	\arrow[curve={height=12pt}, squiggly, no head, from=1-7, to=1-1]
	\arrow["{m^\prime}", from=1-7, to=3-7]
	\arrow[squiggly, no head, from=3-1, to=3-3]
	\arrow["{m^\prime}"', from=3-3, to=5-4]
	\arrow["m", from=3-5, to=5-4]
	\arrow[squiggly, no head, from=3-7, to=3-5]
\end{tikzcd}\]
\caption{	\label{fig:mMoreThanOneDif}
The squiggly lines represent a sequence of elementary semicommutations in the same way $\sim$ does in the body of the text.
}
\end{figure}

	Suppose that $\mid m - m^\prime \mid \geq 1$.
By Lemma \ref{lem:wDopExists} we know that there exists some $w^\dop \in \freemon[U_H \cup \ltr]$ such that $w \sim w^\dop$ and cancellations at both $m$ and $m^\prime$ are possible.
As $\mid m - m^\prime \mid \geq 1$ we know that the two possible cancellations do not overlap and thus that unlike the previous cases it is possible to perform them sequentially (and that either order will produce the same product).
Let $v^\flat, v^\dop, v^\sharp \in \freemon[U_H \cup \ltr]$ be such that $w^\dop \rightarrow_m v^\flat \rightarrow_{m^\prime} v^\dop$ and $w^\dop \rightarrow_{m^\prime} v^\sharp \rightarrow_m v^\dop$.
This means that $v \leftarrow_m w \sim w^\dop \rightarrow_m v^\flat$ and so by Lemma \ref{lem:mCancelSemicommutingSquare} we know that $v \sim v^\flat$.
Similarly we can find that $v^\sharp \sim v^\prime$.
Therefore $v \sim v^\flat \rightarrow_m v^\dop \leftarrow_{m^\prime} v^\sharp \sim v^\prime$ as required.
See Figure \ref{fig:mMoreThanOneDif} for an illustration of this.

\end{proof}

	We use $w \rightsquigarrow v$ to indicate that $w \rightsquigarrow_m v$ for some $m$.

\begin{lemma}	\label{lem:smallSinkExists}

	Suppose we have words $w, v, v^\prime \in \freemon[U_H \cup \ltr]$ in HNN reduced form such that $v \leftsquigarrow w \rightsquigarrow v^\prime$.
Either $v \sim v^\prime$ or there exists $v^\dop \in \freemon[U_H \cup \ltr]$ in HNN reduced form such that $v \rightsquigarrow v^\dop \leftsquigarrow v^\prime$.

\end{lemma}

\begin{proof}

	The hypotheses in the statement of the lemma imply that there are HNN reduced form words $v_1, v_2, w_1, w_2 \in \freemon[U_H \cup \ltr]$ and numbers $m, m^\prime \in \mathbb{N}$ such that $v \sim v_1$, $v_2 \sim v^\prime$, $w_1 \sim w \sim w_2$ and $v_1 \leftarrow_m w_1 \sim w_2 \rightarrow_{m^\prime} v_2$.
The rest follows immediately from Lemma \ref{lem:smallSinkTechnical} and the definition of $\rightsquigarrow$.

\end{proof}

	We can further show that any two words equal in $H^\ast$ possess a, not necessarily unique, ``sink'' word that can be reached by a finite sequence of elmentary cancellations and semicommutations from either starting word.

\begin{lemma}	\label{lem:largeSinkTechnical}

	Let $v_0, v_1, \ldots, v_k \in \freemon[(U_H \cup \ltr)]$ be a set of words in HNN reduced form such that $v_i \leftsquigarrow v_{i+1}$ or $v_i \rightsquigarrow v_{i+1}$, for $0 \leq i < k$.

There exists a sequence of words $v_0^\prime, v_1^\prime, \ldots, v_{k^\prime}^\prime \in \freemon[(U_H \cup \ltr)]$ in HNN reduced form such that
$v_0^\prime \equiv v_0$,
$v_{k^\prime}^\prime \equiv v_k$,
$k^\prime \leq k$,
$v_i^\prime \rightsquigarrow v_{i+1}^\prime$ for $0 \leq i < m$ 
and
$v_i^\prime \leftsquigarrow v_{i+1}^\prime$ for $m \leq i \leq k^\prime$, for some $0 \leq m \leq k^\prime$.

\end{lemma}

\begin{proof}

	If $k = 0$ then the statement is trivial.

	If $k = 1$ then either $v \equiv v_0 \leftsquigarrow v_1 \equiv v^\prime$ or $v \equiv v_0 \rightsquigarrow v_1 \equiv v^\prime$ in which case choosing $v^\dop \equiv v$ or $v^\dop \equiv v^\prime$ respectively would satisfy the necessary conditions.

	If $k = 2$, then either one of $v_0$, $v_1$ and $v_2$ make a suitable choice for $v^\dop$ or we have that $v_0 \leftsquigarrow v_1 \rightsquigarrow v_2$.
In this case by Lemma \ref{lem:smallSinkExists} either there is some suitable $v^\dop$ or $v \sim v^\prime$ in which case we can satisfy the necessary conditions by taking $v^\dop \equiv v$.

	Assume that $k = K + 1 \geq 3$ and that for all sequences with $k \leq K$ arrows the Lemma's statement is satisfied.
This means that we can find a sequence of words $v_0^\prime, \ldots, v_{K^\prime}^\prime \in \freemon[U_H \cup \ltr]$, for $K^\prime \leq K$, such that $v_0 \equiv v_0^\prime$, $v_K \equiv v_{K^\prime}^\prime$ and that there exists some $0 \leq m \leq K^\prime$ such that $v_j \rightsquigarrow v_{j+1}$ if $0 \leq j < m$ and $v_j \leftsquigarrow v_{j+1}$ if $m \leq j \leq K^\prime$.
If $K^\prime < K$ then there is a sequence of $K^\prime + 1 \leq K$ arrows between $v_0$ and $v_{K+1}$ and the Lemma's statement is satisfied by inductive assumption.
So going forward we assume $K^\prime = K$.

	Suppose $v_K \leftsquigarrow v_{K+1}$ then we can say that
\begin{equation*}
	v_0 \rightsquigarrow v_1^\prime \rightsquigarrow v_2^\prime \rightsquigarrow
	\ldots
	\rightsquigarrow v_m^\prime \leftsquigarrow
	\ldots
	\leftsquigarrow v_{K -1}^\prime \leftsquigarrow v_K \leftsquigarrow v_{K+1}
\end{equation*}
and have a sequence of arrows satisfying the necessary conditions.
Otherwise $v_K \rightsquigarrow v_{K+1}$ and we have
\begin{equation*}
	v_0 \rightsquigarrow v_1^\prime \rightsquigarrow v_2^\prime \rightsquigarrow
	\ldots
	\rightsquigarrow v_m^\prime \leftsquigarrow
	\ldots
	\leftsquigarrow v_{K^ -1}^\prime \leftsquigarrow v_K \rightsquigarrow v_{K+1}
\end{equation*}
In this case we can apply Lemma \ref{lem:smallSinkExists} to 
$v_{K^\prime -1}^\prime \leftsquigarrow v_{K^\prime}^\prime \rightsquigarrow v_{K+1}$.
This means either that $v_{K^\prime -1}^\prime \sim v_{K+1}$ or that there is some word $v_K^\dop \in \freemon[(U_H \cup \ltr)]$ in HNN reduced form such that $v_{K-1}^\prime \rsa v_K^\dop \lsa v_{K+1}$.
If the former then we have a sequence of $K^\prime - 1 + 1 = K^\prime \leq K$ arrows between $v_0$ and $v_{K+1}$ (and as above we can satisfy the necessary conditions by inductive assumption).
If the latter then we have that
\begin{equation*}
	v_0 \rsa v_1^\prime \rsa v_2^\prime \rightsquigarrow
	\ldots
	\rightsquigarrow v_j^\prime \leftsquigarrow
	\ldots
	\leftsquigarrow v_{K-1}^\prime \rsa v_K^\dop \lsa v_{K+1}.
\end{equation*}
We can apply the inductive assumption to the subsequence of arrows going from $v_0$ to $v_K^\dop$ in a similar manner to above and get a sequence of the form
\begin{equation*}
	v_0 \rightsquigarrow v_1^\dop \rightsquigarrow v_2^\dop \rightsquigarrow
	\ldots
	\rightsquigarrow v_{m^\prime}^\dop \leftsquigarrow
	\ldots
	\leftsquigarrow v_{K^\dop-1}^\dop \leftsquigarrow v_{K^\dop}^\dop \lsa v_{K+1},
\end{equation*}
for some $m^\prime$, which satisfies the necessary conditions.
For a compiled diagram illustrating the proof see Figure \ref{fig:sinkInduction}.

\begin{figure}
\centering
\adjustbox{scale = 0.95}{
\begin{tikzcd}
	& {v_1} & \ldots & \ldots & \ldots & {v_{K-1}} \\
	{v_0} &&&&&& {v_K} \\
	& {v_1^\prime} & \ldots & {v_m^\prime} & \ldots & {v_{K-1}^\prime} && {v_{K+1}} \\
	&&&&&& {v_{K}^{\prime \prime}} \\
	& {v_1^{\prime \prime}} & \ldots & {v_{m^\prime}^{\prime \prime}} & \ldots & {v_{K-1}^{\prime \prime}}
	\arrow[squiggly, tail reversed, from=1-2, to=1-3]
	\arrow[squiggly, from=1-2, to=2-1]
	\arrow[squiggly, tail reversed, from=1-3, to=1-4]
	\arrow[squiggly, tail reversed, from=1-5, to=1-4]
	\arrow[squiggly, tail reversed, from=1-5, to=1-6]
	\arrow[squiggly, tail reversed, from=1-6, to=2-7]
	\arrow[squiggly, from=2-1, to=3-2]
	\arrow[squiggly, from=2-1, to=5-2]
	\arrow[squiggly, from=2-7, to=3-6]
	\arrow[squiggly, from=2-7, to=3-8]
	\arrow[squiggly, from=3-2, to=3-3]
	\arrow[squiggly, from=3-3, to=3-4]
	\arrow[squiggly, from=3-5, to=3-4]
	\arrow[squiggly, from=3-6, to=3-5]
	\arrow[squiggly, from=3-6, to=4-7]
	\arrow[squiggly, from=3-8, to=4-7]
	\arrow[squiggly, from=4-7, to=5-6]
	\arrow[squiggly, from=5-2, to=5-3]
	\arrow[squiggly, from=5-3, to=5-4]
	\arrow[squiggly, from=5-5, to=5-4]
	\arrow[squiggly, from=5-6, to=5-5]
\end{tikzcd}
}
\caption{	\label{fig:sinkInduction}
	Double ended squiggly arrows indicate an indifference to direction.
This demonstrates a case where the most extensive method is require.
}
\end{figure}
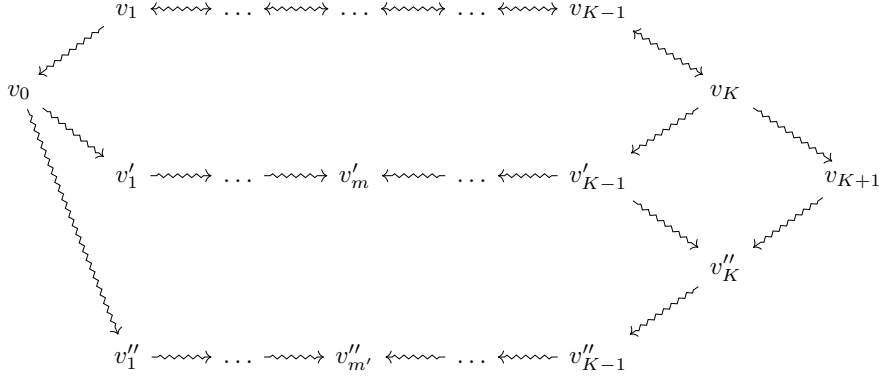

\end{proof}

	We use $w \succsim v$ to indicate that either $w \sim v$ or there is a sequence of words $w_0, w_1, \ldots, w_k \in \freemon[(U_H \cup \ltr)]$ in HNN reduced form such that $w \equiv w_0$, $v \equiv w_k$, $w_i \rightsquigarrow w_{i+1}$ and $k \geq 1$.

\begin{lemma}	\label{lem:largeSinkExists}

	Let $v, v^\prime \in \freemon[U_H \cup \ltr]$ be words which are in HNN reduced form in $H^\ast$.
If $v = v^\prime$ in $H^\ast$ then there exists some word $v^\dop \in \freemon[U_H \cup \ltr]$ such that $v \succsim v^\dop \precsim v^\prime$.

\end{lemma}

\begin{proof}

	By Lemma \ref{lem:equalMeansSeqOfElOps} $v = v^\prime$ in $H^\ast$ implies that there is a finite sequence of elementary operations between them.
This sequence can only contain a finite number of elementary additions and elementary cancellations, thus there is a sequence of words $v_0, v_1, \ldots, v_k \in \freemon[(U_H \cup \ltr)]$ such $v_i \leftsquigarrow v_{i+1}$ or $v_i \rightsquigarrow v_{i+1}$.
The rest follows from Lemma \ref{lem:largeSinkTechnical}.

\end{proof}

	We may also strengthen Lemma \ref{lem:WrsaVmeansWrsaNiceV}.

\begin{lemma}	\label{lem:wordOverBasisSubsetExists}

	Let $v_1, v_k \in \freemon[U_H \cup \ltr]$ be words which are in HNN reduced form in $H^\ast$.
Further let $V_H \subseteq \ouh$ be such that $v_1 \in ( V_H \cup \lttr )^\ast$.
If $v_1 \succsim v_k$ then there exists $v_k^\prime \in ( V_H \cup \lttr )^\ast$ such that $v_k \sim v_k^\prime$.

\end{lemma}

\begin{proof}

	As $v_1 \succsim v_k$ there is a sequence of HNN reduced form words $v_i \in \freemon[U_H \cup \ltr]$, for $1 \leq i \leq k$, such that $v_i \rightsquigarrow_{m_i} v_{i+1}$, for $1 \leq i < k$.

	Let $i = 1$.
As $v_1^\prime \equiv v_1$, then trivially we have $v_1 \sim v_1^\prime$ and $v_1^\prime \in ( V_H \cup \lttr )^\ast$.

	Let $i = K < k$ and suppose we have $v_K^\prime \in ( V_H \cup \lttr )^\ast$ such that $v_K^\prime \sim v_K$.
This means that $v_K^\prime \rightsquigarrow_{m_K} v_{K+1}$.
So by Lemma \ref{lem:WrsaVmeansWrsaNiceV}, we know that there exists $v_{i+1}^\prime \in ( V_H \cup \lttr )^\ast$ such that $v_K^\prime \rightsquigarrow_{m_K} v_{K+1}^\prime$.
Thus we have $v_{K+1} \leftsquigarrow_{m_K} v_K^\prime \rightsquigarrow_{m_K} v_{K+1}^\prime$ which, by Lemma \ref{lem:smallSinkTechnical}, means that $v_{K+1}^\prime \sim v_{K+1}$.

	Therefore by induction we can find $v_k^\prime \in ( V_H \cup \lttr )^\ast$ such that $v_k^\prime \sim v_k$.

\end{proof}

	We shall say that a word $v \in \freemon[U_H \cup \ltr]$ is in its \textit{most reduced form} if it is in HNN reduced form and there is no $v^\prime \in \freemon[U_H \cup \ltr]$ such that $v \rightsquigarrow v^\prime$.
We can then define
\begin{equation*}
	\MRF(w) = \lbrace v \in \freemon[(U_H \cup \ltr)] \mid v \precsim w, \text{ $v$ is in its most reduced form} \rbrace
\end{equation*}
to be the set of words which are most reduced forms of the defining word.
We will now prove this set has a number of properties, which will be useful in the next section.

\begin{lemma}	\label{lem:MRFstuff}

	Let $w, w^\prime \in \freemon[U_H \cup \ltr]$ be words in HNN reduced form.

Then the following hold
\begin{enumerate}
	\item The set $\MRF(w)$ is non-empty;
	\item Let $v_1 \in \MRF(w)$ then $v_2 \in \MRF(w)$ if and only if $v_1 \sim v_2$;
	\item Either $\MRF(w) \equiv \MRF(w^\prime)$ or $\MRF(w) \cap \MRF(w^\prime) = \emptyset$;
	\item The sets $\MRF(w) \equiv \MRF(w^\prime)$ if and only if $w = w^\prime$ in $H^\ast$;
	\item The set $\MRF(w)$ is finite and can be found by algorithm;
	\item Let $V \subseteq \ouh$ be such that $w \in (V \cup \lttr)^\ast$, then there is some word $v^\prime \in (V \cup \lttr)^\ast$ such that $v^\prime \in \MRF(w)$ and this can be found by algorithm.
\end{enumerate}

\end{lemma}

\begin{proof}

	$(1)$:
The word $w$ is made up of a finite concatenation of words from $\lttr$ and $\fmuh$.
Each elementary cancellation reduces the amount of basis words from $\ouh$ used.
Thus every sequence of elementary cancellations and semicommutations can only contain finitely many elementary cancellations.
Thus there must be a most reduced form for $w$.

	$(2)$:
Suppose that $v_1 \sim v_2$.
As $v_2 \sim v_1 \precsim w$ we have that $v_2 \precsim w$.
Thus if $v_2 \notin \MRF(w)$ there must be some $v_3 \in \freemon[U_H \cup \ltr]$ such that $v_2 \rightsquigarrow v_3$.
However this implies that $v_1 \sim v_2 \rightsquigarrow v_3$ and hence that $v_1 \rightsquigarrow v_3$, which contradicts $v_1 \in \MRF(w)$.
Therefore $v_2 \in \MRF(w)$.

	Suppose that $v_2 \in \MRF(w)$.
That $v_2 \precsim w \succsim v_1$ means that there is a sequence of elementary operations between them.
By Lemma \ref{lem:equalMeansSeqOfElOps} this means that $v_1 = v_2$ and thus by Lemma \ref{lem:largeSinkExists} there is some $v_3$ such that $v_2 \succsim v_3 \precsim v_1$.
By assumption $v_2$ is in most reduced form, therefore $v_2 \succsim v_3$ implies that $v_2 \sim v_3$ as no cancellations are possible.
Similarly $v_3 \sim v_1$ and thus $v_1 \sim v_2$.

	$(3)$:
	Suppose $v \in \MRF(w) \cap \MRF(w^\prime)$.
If $v^\prime \in \MRF(w^\prime)$ then, as $v \in \MRF(w^\prime)$, by $(2)$ we know that $v \sim v^\prime$.
Thus by a second application of $(2)$, as $v \in \MRF(v)$, we know that $v^\prime \in \MRF(w)$.
Therefore $\MRF(w^\prime) \subseteq \MRF(w)$.
Dually $\MRF(w) \subseteq \MRF(w^\prime)$ and so $\MRF(v) = \MRF(v^\prime)$.

	Thus we have shown that any overlap forces equality between the sets.

	$(4)$:
Suppose that $\MRF(w) \equiv \MRF(w^\prime)$.
Let $v \in \MRF(w) \equiv \MRF(w^\prime)$.
This means that $w \succsim v \precsim w^\prime$.
Thus there is a sequence of elementary operations between $w$ and $w^\prime$ and so by Lemma \ref{lem:equalMeansSeqOfElOps} $w = w^\prime$ in $H^\ast$.

	Suppose that $w = w^\prime$ in $H^\ast$.
By Lemma \ref{lem:largeSinkExists} there exists some $w^\dop \in \freemon[U_H \cup \ltr]$ such that $w \succsim w^\dop \precsim w^\prime$.
We know by $(1)$ there is some word $v \in \MRF(w^\dop)$.
As $v \precsim w^\dop$ it follows that $v \precsim w$ also.
Thus as we already know that $v$ is in most reduced form it follows that $v \in \MRF(w)$.
Similarly we may also deduce that $v \in \MRF(w^\prime)$.
By $(3)$ we know that these sets are either equal or disjoint, therefore as their intersection is non-empty we know that $\MRF(w) \equiv \MRF(w^\prime)$.

	$(5)$:
Let $w_1 \equiv w$.
By Lemma \ref{lem:basisWordShift} we know that the set $W_1 = \lbrace w_1^\prime \mid w_1^\prime \sim w_1 \rbrace$ is finite and can be found by algorithm.
If we find some $w_1^\prime \in W_1$ where an elementary cancellation is possible then we perform that cancellation and call the result $w_2$.

	If we repeat this process until we find some $w_k$ such that no member of $W_k = \lbrace w_k^\prime \mid w_k^\prime \sim w_k \rbrace$ has any possible elementary cancellations (this must happen as $w$ is of finite length and each $w_{i+1}$ is strictly shorter than the preceding $w_{i}$).
As $w_i \rightsquigarrow w_{i+1}$ we have that $v_k \precsim v$ and so $w_k \in \MRF(w)$.
Moreover, by $(2)$ we have that $\MRF(w) = W_k$.

	$(6)$:
By $(1)$ there is some $v \in \MRF(w)$.
As $v \precsim w$ by Lemma \ref{lem:wordOverBasisSubsetExists} there exists some $v^\prime \in (V \cup \lttr)^\ast$ such that $v^\prime \sim v$.
By $(2)$ this means that $v^\prime \in \MRF(w)$.
By $(5)$ there are a finite number of words in $\MRF(w)$ and these can be found by algorithm.
When given a finite set of words written over $U_H \cup \lttr$ we can decide if any are written over $V \cup \lttr$ by algorithm.
Therefore we can find $v^\prime$ by algorithm.

\end{proof}

	Taken together these results mean that for any word that is already in HNN reduced form we can find a set of words which collectively act as something like a normal form.
In the next section we will see how this can be applied to solve membership problems.

\section{Submonoids with decidable membership}

	In this section we will show that certain submonoids have decidable membership within the family of HNN extensions described in the previous section.
Before we do that we need to show the following.

\begin{lemma}	\label{lem:QisPcapH}

	Let $H^\ast = H \ast_{t, \phi : A \rightarrow B}$ be a HNN extension of a free group $H = \gpPres$.
Further let $Q \leq H$ be a monoid and let $P = \subMon{Q \cup \lttr} \leq H^\ast$.
Then $Q = P \cap H$ if and only if $\phi(Q \cap A) = Q \cap B$.

\end{lemma}

\begin{proof}

	``$\Rightarrow$'':
Suppose that $Q = P \cap H$.
Let $q \in Q \cap A$.
It is clear from $q \in Q$ and the definition of $P$ that $t^{-1} q t \in P$.
Additionally $q \in A$ implies that $t^{-1} q t = \phi(q) \in B$.
So $\phi(q) \in P \cap B = P \cap (H \cap B) = (P \cap H) \cap B = Q \cap B$, which means that $\phi(Q \cap A) \subseteq Q \cap B$.
A dual argument gives that $\phi^{-1}(Q \cap B) \subseteq Q \cap A$, which is equivalent to $\phi(Q \cap A) \supseteq Q \cap B$ and therefore $\phi(Q \cap A) = Q \cap B$.

	``$\Leftarrow$'':
	Suppose that $\phi(Q \cap A) = Q \cap B$.
Let $p \in P \cap H$.
We know by the definition of $P$ that we can write this in the form $q_0 t^{\varepsilon_1} q_1 \ldots t^{\varepsilon_n} q_n$, where $q_i \in Q$ for $0 \leq n$ and $\varepsilon_i \in \lbrace -1, 1 \rbrace$ for $1 \leq i \leq n$.
Further as $p \in H$ we know that the sum of $\varepsilon_i$ must be $0$ and that the form is only HNN reduced if $n = 0$.
We know that if a word is not HNN reduced then there is either a sequence of the form $t^{-1} a t$ where $a \in A$ or one of the form $t b t^{-1}$ where $b \in B$.
This means that there is either an $i$ such that $\varepsilon_i = -1$, $q_i \in A$ and $\varepsilon_{i+1} = 1$ or an $i$ such that  $\varepsilon_i = 1$, $q_i \in B$ and $\varepsilon_{i+1} = -1$.

	If the first is true then $t^{\varepsilon_i} q_i t^{\varepsilon_{i+1}} = \phi(q_i)$, this being a well defined use of $\phi$ due to $q_i \in A$.
Moreover as $q_i \in Q \cap A$ we know that $\phi(q_i) \in Q \cap B$.
This means that we can write $p$ in the form
\begin{equation*}
	q_0 t^{\varepsilon_1}
	\ldots
	t^{\varepsilon_{i-1}} q_{i-1} \phi(q_i) q_{i+1} t^{\varepsilon_{i+2}} 
	\ldots
	t^{\varepsilon_n} q_n
\end{equation*}
which is composed of elements belonging to $Q$ and $\lttr$ only and has two fewer occurences of $t$ and $t^{-1}$ collectively.
A dual argument can be made for a similar reduction in the second case.

	This means that if $n \neq 0$ then we can inductively reduce $p$ to a form where $n = 0$ while preserving that the non-$t$ parts belong to $Q$, so $p \in Q$ and therefore $P \cap H \subseteq Q$.
That $Q \subseteq P \cap H$ is obvious from the definitions and so we can conclude that $Q = P \cap H$.

\end{proof}

	We can now further deduce the following, which is one of the main results of this paper.

\begin{theorem}	\label{thm:submonMemInHast}

	Let $H^\ast = H \ast_{t, \phi : A \rightarrow B}$ be a HNN extension of a free group $H = \gpPres$ with free basis $U_H \subset \freemon$.
Further let this free basis be such that there are $U_A, U_B \subset \ouh$ where $A = \subGp{U_A}$, $B = \subGp{U_B}$ and $\phi(\oua) = \oub$.
Finally let $Q$ be a submonoid of $H$ generated by some $U_Q \subseteq \ouh$.

	Membership in the submonoid $P =  \subMon{ Q \cup \lttr }$ within $H^\ast$ is decidable if $\phi(Q \cap A) = Q \cap B$.

\end{theorem}

\begin{proof}

	Suppose we have a word $w \in \freemon[X \cup \ltr]$.
By Lemma \ref{lem:redFormByAlg} we can algorithmically find a word equal to $w$ in $H^\ast$ which is in HNN reduced form.
Therefore we can assume without loss of generality that $w$ in HNN reduced form.
We can also assert without loss of generality that the parts of $w$ written over $\freemon$ are written over the free basis $\ouh$.

	If the element in $H^\ast$ represented by $w$ belongs to $P$ then there is some word $w^\prime \in (U_Q \cup \lttr)^\ast$ which is equal to $w$ in $H^\ast$.
We saw in the proof of Lemma \ref{lem:QisPcapH} if we can find a HNN reduced form of $w^\prime$ while preserving that the $H$ parts are written over $U_Q$.
We can, therefore, assume without loss of generality that this $w^\prime$ is also in HNN reduced form, as well as written over $U_Q \cup \lttr$.

	By part $(4)$ of Lemma \ref{lem:MRFstuff} that $w = w^\prime$ in $H^\ast$ implies that $\MRF(w) = \MRF(w^\prime)$.
By part $(6)$ of the same Lemma that $w^\prime \in (U_Q \cup \lttr)^\ast$ implies that there is some $v \in \MRF(w^\prime) = \MRF(w)$ such that $v \in (U_Q \cup \lttr)^\ast$.
As $v \in \MRF(w)$ we know that $v \precsim w$ and that therefore there is a sequence of elementary operations between $v$ and $w$, by Lemma \ref{lem:equalMeansSeqOfElOps} this means that $v = w$ in $H^\ast$.
By the definition of $P$, we can see that $v \in P$ and thus that $w \in P$.

	By part $(5)$ of the Lemma all the words in $\MRF(w)$ can be found by algorithm.
As we can determine if a word written over $(U_H \cup \lttr)$ is also written over $(U_Q \cup \lttr)$ by algorithm, we may algorithmically determine whether the set $\MRF(w) \cap (U_Q \cup \lttr)^\ast$ is non-empty.
Hence we can decide whether $w \in P$ within $H^\ast$.
\end{proof}

\begin{remark}

	We may observe that this theorem almost fulfils the requirements of Theorem C of Dolinka and Gray \cite{DolGra21} as, by Lemma \ref{lem:QisPcapH}, $Q = P \cap H$.
However the key difference is that their result requires that $A \cup B \subseteq M$ and so it cannot be used to reproduce this one.

\end{remark}

\begin{remark}
	
Though the next section is focused on applying these results to one-relator HNN extensions of one-relator groups, it is important to note that nothing in this section nor the preceding one is restricted to one-relator HNN extensions nor to HNN extensions of one-relator groups.
In all cases we have only assumed that the HNN extension and its basis are finitely presented (in addition, of course, to whatever the various lemmas and theorems explicitly require).

\end{remark}

\section{Application to the Prefix Membership Problem}

	In this section we show how the results of the previous section can be applied to the prefix membership problem of a family of groups.
We can then use this to decide the word problem of the inverse monoids with presentations corresponding to the family of groups by the following result of Ivanov, Margolis and Meakin \cite[Theorem 4.1]{IMM01}.
We note that the family we identify in Theorem \ref{thm:wtwWords} below has the property of being defined by defined by a single $t$-sum zero word whose prefixes are a mixture of $t$-sum positive and $t$-sum negative, and that they have decidable prefix membership problem cannot be shown by the methods in Dolinka and Gray \cite{DolGra21}.

\begin{theorem}	\label{thm:oneRelPMPgivesWP}

	Let $M = \invPres[r = 1]$ be an inverse monoid and $G = \gpPres[r = 1]$ be its maximal group image.
If $r$ is cyclically reduced and $G$ has decidable prefix membership problem then $M$ has decidable word problem.

\end{theorem}

	We can combine this with Theorem \ref{thm:submonMemInHast} and Theorem \ref{thm:HNNextByRho} to get our final result.

\begin{theorem}	\label{thm:wtwWords}

	Let $M = \invPres[wt^{-2 \sigma_t(w)}w = 1][X \cup \ltr]$ be an inverse monoid and let $G = \gpPres[wt^{\sigma_t(w)}w = 1][X \cup \ltr]$ be its maximal group image,
where the word $w \in \freemon[X \cup \ltr]$ meets the following requirement:
\begin{itemize}
	\item That $w \equiv x_0^{\varepsilon_0} t^{n_1} x_1^{\varepsilon_1} \ldots t^{n_k} x_k^{n_k}$, such that $n_i \in \mathbb{Z}$ for $1 \leq i \leq k$, $\varepsilon_i \in \lbrace -1, 1 \rbrace$ and $x_i \in X$ for $0 \leq i \leq k$ and $x_i \not\equiv x_j, x_j^{-1}$ if $i \neq j$;
	\item That $\sigma_t(w) = n_1 + \ldots + n_k \neq 0$.
\end{itemize}

Both the prefix membership problem for $G$ and the word problem for $M$ are decidable.

\end{theorem}

\begin{proof}

	We will assume throughout that $\sigma(w) \geq 0$ and that all $\varepsilon_i = 1$, we may do this without loss of generality as all words are only a relabelling away from such a case.

	It may be seen that the sole relator word of $M$ has a unital factorisation into $w$ and $t$.
This means by Lemma \ref{lem:unitalIsCon} that in $G$ the factorisation is conservative and so the prefix monoid of $G$ can take the form
\begin{equation*}
	P = \subMon{\pref(w),t,t^{-1}}.
\end{equation*}

	All prefixes of $w$ have the form $p \equiv x_0^{\varepsilon_0} t^{n_1} x_1^{\varepsilon_1} \ldots x_i t^{j}$ where $i \leq k$ and either $0 \leq j \leq n_{i+1}$ if $n_{i+1} \geq 0$ or $n_{i+1} \leq j \leq 0$ if $n_{i+1} \leq 0$.
However as we have both $t$ and $t^{-1}$ in our generating set for $P$ any prefixes of $w$ ending in $t$ or $t^{-1}$ are extraneous.
So we may write
\begin{equation*}
	P = \subMon{x_0, x_0 t^{n_1} x_1, \ldots, x_0 t^{n_1} x_1 \ldots t^{n_k} x_k, t, t^{-1}}.
\end{equation*}

	The overall $t$-exponent sum of $G$'s sole relator is $0$, therefore we can apply Theorem \ref{thm:HNNextByRho} to $G$ and acquire a new group $H$ which $G$ can be viewed as a HNN extension of.

	The group $H$ takes the form
\begin{equation*}
	H = \gpPres[x_{0,s_0} x_{1,s_1} \ldots x_{k,s_k} x_{0,s_0+\sigma(w)} x_{1,s_1+\sigma(w)} \ldots x_{k,s_k+\sigma(w)} = 1][\Xi_r]
\end{equation*}
where $s_i = n_1 + \ldots + n_i$ and $\Xi_r$ is the alphabet consisting of all $x_{i,j}$ such that $0 \leq i \leq k$ and $s_i \leq j \leq s_i + \sigma(w)$.

	We claim that every $x_{i,j}$ in the defining relation of $H$ is distinct.
As the $x_i$ in $w$ is defined as distinct letters, we know that for all $0 \leq i \leq k$, the letters $x_{i, s_i}$ are distinct from each other.
Likewise the letters $x_{i, s_i + \sigma(w)}$ are distinct from each other.
Moreover as $\sigma(w) \neq 0$ these two sets of letters are disjoint.

	This group is such that if we take $\phi$ to be the mapping $\phi(x_{i,j}) = x_{i,j+1}$, 
$A$ to be the subgroup generated by $\Xi_r \setminus \lbrace x_{i,s_i} \text{ for } 0 \leq i \leq k \rbrace$ and 
$B$ to be the subgroup generated by $\Xi_r \setminus \lbrace x_{i,s_i + \sigma(w)} \text{ for } 0 \leq i \leq k \rbrace$
then the HNN extension $H^\ast = H \ast_{t, \phi: A \rightarrow B}$ is isomorphic to $G$ under the mapping $x_{i,j} \mapsto t^{-j}x_it^j$.

	Applying the preimage of such a mapping to our most recent generating set for $P$ we get
\begin{equation}
	P^\ast = \subMon{x_{0,s_0}, \, x_{0, s_0} x_{1, s_1}, \,
	\ldots, \,
	x_{0,s_0} x_{1,s_1} \ldots x_{k,s_k}, t, t^{-1} }
\end{equation}
such that $P^\ast$ within $H^\ast$ is isomorphic to $P$ within $G$.

	We define $\gamma_{i,j} = x_{0,s_0+j} x_{1,s_1+j} \ldots x_{i,s_i+j}$ for $0 \leq i \leq k$ and $0 \leq j \leq \sigma(w)$, we shall refer to the set of all such $\gamma_{i,j}$ as $\Gamma$.

	We claim that $U_H = \Gamma \setminus \lbrace \gamma_{k, \sigma(w)} \rbrace$ is a free basis of $H$.
	First we note that $\gamma_{k,0}^{-1} = \gamma_{k, \sigma(w)}$ in $H$ and so we can immediately recover the discarded element of $\Gamma$.
Secondly, $x_{0, s_0 + j} = \gamma_{0,j}$ and $x_{i, s_i + j} = \gamma_{i-1,j}^{-1} \gamma_{i,j}$ for $0 < i \leq k$.
Therefore $\overline{\Gamma \setminus \lbrace \gamma_{k, \sigma(w)} \rbrace}$ generates all of $\Xi_r$ and consequently all of $H$.
Finally, as there is an obvious one-to-one correspondence between the elements of $\Gamma$ and $\Xi_r$ the two sets are of the same size.
This means that $U_H = \Gamma \setminus \lbrace \gamma_{k, \sigma(w)} \rbrace$ is of size $\mid \Xi_r \mid - 1$ which is the rank of $H$ as a free group and so by Theorem \ref{thm:FGdefByRank} $U_H$ is a free basis of $H$.

	We can now rewrite our previous expression over the new basis to produce
\begin{equation*}
	P^\ast = \subMon{\gamma_{0,0}, \gamma_{1,0}, \ldots, \gamma_{k, \sigma(w)}, t, t^{-1} }.
\end{equation*}
As $t^{-1} \gamma_{i,j} t= \gamma_{i,j+1}$ for $0 \leq j < \sigma(w)$, the generators above are sufficient to produce the rest of $\Gamma$.
Thus we can write
\begin{equation*}
	P^\ast = \subMon{ \Gamma \cup \lttr }
\end{equation*}
as we may add superfluous generators without consequence.
We may then find the form
\begin{equation*}
	P^\ast = \subMon{ \left( \Gamma \setminus \lbrace \gamma_{k, \sigma(w)} \rbrace \right) \cup \lbrace \gamma_{k,0}^{-1} \rbrace \cup \lttr }
\end{equation*}
by making the substitution of $\gamma_{k,0}^{-1}$ for $\gamma_{k, \sigma(w)}$ as these are equal in $H$.

	We now let $U_Q = U_H \cup \lbrace \gamma_{k,0}^{-1} \rbrace = \left( \Gamma \setminus \lbrace \gamma_{k, \sigma(w)} \rbrace \right) \cup \lbrace \gamma_{k,0}^{-1} \rbrace$ and $Q = \subMon{U_Q}$.
We seek to show that $\phi(Q \cap A) = Q \cap B$.
In $H$ we know that $\gamma_{k,\sigma(w)} = \gamma_{k,0}$, thus $U_Q = \Gamma$ in $H$.

	All the defined mappings of $\phi$ and $\phi^{-1}$ send individual elements of $\Gamma$ to other individual elements of $\Gamma$, and thus we may deduce that $\phi(\Gamma \cap A) = \Gamma \cap B$.
Furthermore as $U_Q$ is a relabelling $\Gamma$ this means that $\phi(U_Q \cap A) = U_Q \cap B$.
By the first part of Lemma \ref{lem:monVMgpVG} we can see that by setting $V_M = U_Q$ and $V_G = U_A$ that $U_Q \cap A = U_Q \cap \oua$, dually we can see that $U_Q \cap B  = U_Q \cap \oub$.
Further by the second part of the same Lemma we see that any element of $q_a \in Q \cap A = \subMon{U_Q} \cap \subGp{U_A}$ can be written as a word $w_a \in (U_Q \cap  \oua)$.
This means that $\phi(q_a)$ is sent to an element which can be written over
\begin{equation*}
	\phi(U_Q \cap \oua) = \phi(U_Q \cap A) = U_Q \cap B = U_Q \cap \oub.
\end{equation*}
We can then use the second part of Lemma \ref{lem:monVMgpVG} to say that a word written over $U_Q \cap \oub$ represents an element of $Q \cap B$ in $H$.
Thus we have shown that $\phi(Q \cap A) = Q \cap B$.

	We take the sets of words $U_A = U_H \setminus \lbrace \gamma_{0, \sigma(w)}, \gamma_{1, \sigma(w)}, \ldots, \gamma_{k-1, \sigma(w)} \rbrace$ and $U_B = U_H \setminus \lbrace \gamma _{0,0}, \gamma_{1,0}, \ldots, \gamma_{k-1,0} \rbrace$ (recall that $\gamma_{k,0} = \gamma_{k, \sigma(w)}^{-1}$ and that therefore we can apply $\phi^{-1}$) to be the bases for the groups $A$ and $B$ respectively.
Thus we have a setup which fulfills all the conditions of Theorem \ref{thm:submonMemInHast} and therefore membership in $P^\ast$ within $H^\ast$ is decidable.
This is equivalent to being able to decide membership in $P$ within $G$ and so $G$ has decidable prefix membership problem.

	Further as $w t^{-2 \sigma(w)} w$ is cyclically reduced we can apply Theorem \ref{thm:oneRelPMPgivesWP} and thus decide the word problem for $M$.

\end{proof}

\begin{example}

	An example of the class described in Theorem \ref{thm:wtwWords} is the group $\gpPres[bt^{-1}a t^2 bt^{-1}a = 1][a,b,t]$.
The reason for particular interest in this group and its prefix monoid is that Dolinka and Gray \cite{DolGra21} noted it as example to which they could not apply the methods they had developed in that paper for assessing the prefix membership problem and thus left open the question of whether the problem was indeed decidable.

	Under Theorem \ref{thm:wtwWords} this can be viewed as a HNN extension of the group $\gpPres[b_0 a_1 b_{-1} a_0 = 1][a_0, a_1, b_{-1}, b_0]$.
Using the language of the theorem we would say that
$\Gamma = \lbrace b_0, b_0 a_1, b_{-1}, b_{-1} a_0 \rbrace$,
$U_H = \lbrace b_0, b_0 a_1, b_{-1} \rbrace$,
$U_A = \lbrace b_0 a_1, b_{-1} \rbrace$,
$U_B = \lbrace b_0, b_0 a_1 \rbrace$
and $U_Q = \lbrace b_0, b_0 a_1, b_{-1}, a_1^{-1} b_0^{-1} \rbrace$.
Hence by Theorem \ref{thm:wtwWords} we can conclude that this group has decidable prefix membership problem and thus that the corresponding one-relator inverse monoid has decidable word problem.

\end{example}

	This all raises two questions.
The first is whether we can extend the statement of Theorem \ref{thm:wtwWords}.

\begin{question}

	Do all one-relator HNN extensions of free groups have decidable prefix membership problem?

\end{question}

	The second question is whether the machinery developed in Sections $3$ and $4$ can be taken further.
The class of group discussed in Section $3$ seem to have a certain ``freeness'', therefore we may ask:

\begin{question}

	Let $H^\ast = H \ast_{t, \phi: A \rightarrow B}$ be a HNN extension of a finitely generated free group $H$ with free basis $U_H$.
Further let there be $U_A, U_B \subseteq U_H$ such that $\subGp{U_A} = A$, $\subGp{U_B} = B$ and $\phi(\oua) = \oub$.
Does $H^\ast$ have decidable submonoid membership problem?

\end{question}

\bibliography{tsz-citations}{}

\begin{thebibliography}{10}

\bibitem{Ben69}
Mich\`ele Benois.
\newblock Parties rationnelles du groupe libre.
\newblock {\em C. R. Acad. Sci. Paris S\'er. A-B}, 269:A1188--A1190, 1969.

\bibitem{BKS87}
R.~G. Burns, A.~Karrass, and D.~Solitar.
\newblock A note on groups with separable finitely generated subgroups.
\newblock {\em Bull. Austral. Math. Soc.}, 36(1):153--160, 1987.

\bibitem{DolGra21}
Igor Dolinka and Robert~D. Gray.
\newblock New results on the prefix membership problem for one-relator groups.
\newblock {\em Trans. Amer. Math. Soc.}, 374(6):4309--4358, 2021.

\bibitem{IMM01}
S.~V. Ivanov, S.~W. Margolis, and J.~C. Meakin.
\newblock On one-relator inverse monoids and one-relator groups.
\newblock {\em J. Pure Appl. Algebra}, 159(1):83--111, 2001.

\bibitem{Law98}
Mark~V. Lawson.
\newblock {\em Inverse semigroups}.
\newblock World Scientific Publishing Co., Inc., River Edge, NJ, 1998.
\newblock The theory of partial symmetries.

\bibitem{LynSch01}
Roger~C. Lyndon and Paul~E. Schupp.
\newblock {\em Combinatorial group theory}.
\newblock Classics in Mathematics. Springer-Verlag, Berlin, 2001.
\newblock Reprint of the 1977 edition.

\bibitem{Mil71}
Charles~F. Miller, III.
\newblock {\em On group-theoretic decision problems and their classification},
  volume No. 68 of {\em Annals of Mathematics Studies}.
\newblock Princeton University Press, Princeton, NJ; University of Tokyo Press,
  Tokyo, 1971.

\bibitem{Pet84}
Mario Petrich.
\newblock {\em Inverse semigroups}.
\newblock Pure and Applied Mathematics (New York). John Wiley \& Sons, Inc.,
  New York, 1984.
\newblock A Wiley-Interscience Publication.

\bibitem{War24}
Jonathan Warne.
\newblock The word problem of finitely presented special inverse monoids via
  their groups of units (arxiv:2412.03264), 2024.

\bibitem{Wei15}
Armin Weiß.
\newblock {\em On the complexity of conjugacy in amalgamated products and HNN
  extensions}.
\newblock Dissertation, Universität Stuttgart, Stuttgart, 2015.

\end{thebibliography}
\bibliographystyle{plain}

\end{document}